\documentclass[11pt]{amsart}
\usepackage{amsfonts,amsmath,amssymb,amsthm}

\newtheorem{theorem}{Theorem}[section]
\newtheorem{lemma}[theorem]{Lemma}

\newcommand{\av}{\ensuremath{\text{Av}_{n}}}
\begin{document}

\title{An erd\H{o}s-Szekeres-Type Result on Permutation Patterns}
\author{Nicholas Van Nimwegen}

\maketitle

\begin{abstract}
    We construct an infinite set of patterns $(p,q)$ where $p$ is an almost increasing permutation and $q$ is an almost decreasing permutation, such that the exponential growth rate of permutations avoiding $(p,q)$ is 2. This extends a classical result of Erd\H{o}s and Szekeres.
\end{abstract}

\begin{section}{Introduction}

Since first systemically studied by Simon and Schmidt in 1985, the field of permutation patterns has been an active area of mathematics. One of the most common tools used by mathematicians in this area has been the idea of \textit{Wilf-equivalence}, where two patterns $\sigma$ and $\tau$ are said to be Wilf-equivalent if $\av(\sigma_{k}) = \av(\tau)$ for all $n \geq 0$. All permutations have a number of natural Wilf-equivalences under the reverse, complement, and inverse actions, forming a set of symmetries isomorphic to the symmetries of the square $D_{8}$. From this, it is perhaps a natural question to ask how many permutations avoid some subset of symmetries of a given pattern $\sigma$.

There is at least one result that suggests this question may be an interesting one to ask. For a collection of patterns $Q$, let the \textit{exponential growth rate} of permutations avoiding $Q$ be defined as $L(Q) = \lim_{n \rightarrow \infty}\sqrt[n]{\av(Q)}$. In general, the exponential growth rate of permutations avoiding a pattern (or collection of patterns) $Q$ seems to be related to the length of the patterns being avoided. In fact, for a single pattern $\sigma$ of length $k$, it was proven in \cite{k} that $L(\sigma) \geq (k-1)^{2}/e^{3}$. However, the famous Erd\H{o}s-Szekeres theorem states that for all $k \geq 1$, $\av(12\dots k,k(k-1)\dots 1) = 0$ for $n \geq (k-1)^{2}$, and therefore $L(12\dots k,k(k-1)\dots1) =0$. This suggests that avoiding some symmetries of a single pattern $\sigma$ can change the problem of avoidance in such a way that only the structure of the pattern matters, not its length.

We prove a similar result, where avoiding a particular pattern and its reverse seems to be rather difficult, related to growth rates. In particular, we prove the following theorem:
\begin{theorem}
    Let $k \geq 3$, $\sigma_{k} = 12\dots (k-2)k(k-1)$. Then $L(\sigma_{k},\sigma_{k}^{rev}) = 2$.
\end{theorem}
Note that the growth rate has no reliance at all on the length of the patterns being avoided (except for the condition that they be of length at least 3). We can also obtain a generalization, where the lengths of the two patterns can be different sizes:
\begin{theorem}
    Let $k,s \geq 3, \sigma = 12\dots (k-2)k(k-1), \tau = (s-1)s(s-2)\dots 21$. Then $L(\sigma,\tau) = 2$.
\end{theorem}
We conclude with several comments on other sets of patterns, highlighting some other cases where avoiding a pattern and its reverse seems to give some amount of independence from pattern length.

\end{section}

\begin{section}{Proof of Theorem 1.1}

We will begin with stating several lemmas that will be essential for our proof. The first is a result of Reifegerste \cite[Cor 4.3]{reif}:

\begin{lemma}
    Let $k \geq 3$. Then $p_{k}(n)=\av(132,k\dots 21) = \frac{1}{n}\sum_{i=1}^{k-1}{n \choose i}{n \choose i-1}$. In particular, this is a polynomial of degree at most $2k-2$ in $n$.
\end{lemma}

As $p_{k}(n)$ is a polynomial in $n$ of degree $2k-2$, there is some constant $\beta_{k}$ such that $\av(132,k\dots 21) \leq \beta _{k}n^{2k-2}$ for all sufficiently large $n$.

The second is a well-known result, which we prove here for simplicity:

\begin{lemma}
    $\av(132,231) = 2^{n-1}$, and therefore $L(132,231) = 2$.
\end{lemma}

\begin{proof}
    Consider $\pi \in \av(132,231)$. Note that if $\pi_{i} < \pi_{j}$ for $i < j$, then $\pi_{k} > \pi_{j}$ for all $j < k$, else we get a 132 or 231 pattern. Therefore, all entries before the first ascent must be decreasing, and all entries after the first ascent must be increasing. Since any entry following the entry 1 must be an ascent, we can partition all non-1 entries into two sets $A,B$, and place all entries of $A$ before 1 in a decreasing sequence, and all entries of $B$ after 1 in an increasing sequence. There are $2^{n-1}$ ways to form $A,B$, and thus $\av(132,231) = 2^{n-1}$.
\end{proof}

We will also need several definitions for our proof. Let $\pi \in S_{n}$ be a permutation of length $n$. We say $\pi_{i}$ has \textit{rank} $j$ if the longest increasing subsequence ending at $\pi_{i}$ is of length $j$. We similarly say $\pi_{i}$ has \textit{rev-rank} $j$ if the longest decreasing subsequence starting at $\pi_{i}$ has length $j$. Note that $\pi_{i}$ having rev-rank $j$ is equivalent to $\pi^{rev}_{n+1-i}$ having rank $j$.

We are now ready to begin our proof.

\begin{proof}
    First, we show $L(\sigma_{k},\sigma_{k}^{rev}) \geq 2$. If $\pi$ avoids 132, it must necessarily avoid $\sigma_{k}$, since $\sigma_{k}$ contains a 132 in its final 3 entries. Similarly, if $\pi$ avoids 231 it must also avoid $\sigma_{k}^{rev}$. This means $\av(132,231) \subseteq \av(\sigma_{k},\sigma_{k}^{rev})$, and therefore $2=L(132,231) \leq L(\sigma_{k},\sigma_{k}^{rev})$.

    To show $L(\sigma_{k},\sigma_{k}^{rev}) \leq 2$, we will construct a set of permutations $P_{n}$ of length $n$ that contains $\av(\sigma_{k},\sigma_{k}^{rev})$ that has exponential growth rate 2. As a subset of this set, $\av(\sigma_{k},\sigma_{k}^{rev})$ must be smaller, and thus have growth rate at most 2.

    For every permutation $\pi$, we can define 3 partial permutations $A,B,C$. We define $A$ as all entries with rank and rev-rank at most $(k-2)$, $B$ as all entries with rank at least $(k-1)$, and $C$ will be all entries with rev-rank at least $(k-1)$.

    Note that for a permutation avoiding $\big(\sigma_{k},\sigma_{k}^{rev}\big)$, an entry cannot belong to both $B$ and $C$. If there was such an entry $\pi_{i}$, there would necessarily be $b < i < c$ such that $\pi_{b},\pi_{c} < \pi_{i}$ and $\pi_{b}$ (resp. $\pi_{c}$) is of rank (resp. rev-rank) at least $(k-2)$. Then if $\pi_{b} < \pi_{c}$ we would have a $\sigma_{k}$ pattern, and if $\pi_{b} > \pi_{c}$, we would have a $\sigma_{k}^{rev}$ pattern. Thus, $A,B,C$ partition all entries of $\pi$.

    By definition, $A$ must avoid $\big(12\dots (k-1), (k-1)\dots 21\big)$. If $\pi$ avoids $\big(\sigma_{k},\sigma_{k}^{rev}\big)$, then $B$ must avoid $(k-1)\dots 21$ and $132$, the former by the argument above. For the latter, assume $\pi_{i},\pi_{j},\pi_{k}$ form a 132 in $B$. Then by definition of $B$, there must be a length $(k-1)$ increasing subsequence ending in $\pi_{i}$. Then that increasing subsequence together with $\pi_{j},\pi_{k}$ must form a $\sigma_{k}$ pattern. A similar argument holds for $C$ and the patterns $\big(12\dots(k-1),231\big)$, reversing all patterns.

    Let $P_{n}$ be the set of permutations of length $n$ that can be partitioned into subsequences $A,B,C$ such that $A$ avoids $\big(12\dots (k-1),(k-1)\dots 21\big)$, $B$ avoids $\big(132, (k-1)\dots 21\big)$, and $C$ avoids $\big(231, 12\dots (k-1)\big)$. Every permutation that avoids $\big(\sigma_{k},\sigma_{k}^{rev}\big)$ will be in $P_{n}$. Note that these are not sufficient conditions to avoid $\big(\sigma_{k},\sigma_{k}^{rev}\big)$, but they are necessary. For example, the permutation $12\dots (k-1)k$ obviously contains $\sigma_{k}$, but $A = 12\dots (k-2)$, $B = (k-1)k$, $C = \emptyset$, fulfilling our conditions. Therefore, $P_{n}$ will be a larger set than $\av(\sigma_{k},\sigma_{k}^{rev})$.

    In addition, the way we will construct $P_{n}$ will not ensure $A,B,C$ have the rank/rev-ranks that we want. This will result in overcounting, since we are likely to count the same permutation multiple times. However, this is not an issue, since we are using $P_{n}$ as an upper bound.

    Let $\pi \in P_{n}$, $s$ be the size of $A$ in $\pi$. By the Erd\H{o}s-Szekeres theorem, we know $s \leq (k-2)^{2}$, and let $a_{s}$ be the number of permutations avoiding $12\dots (k-1),(k-1)\dots 21$ of length $s$. Then there are at most $a_{s}{n \choose s}^{2}$ ways to define $A$, where ${n \choose s}$ chooses the indices and values of $\pi$ that are part of $A$, and $a_{s}$ chooses the order. Since $s \leq (k-2)^{2}$, ${n \choose s}$ is a polynomial in $n$ of degree at most $(k-2)^{2}$, which is fixed for each $\sigma_{k}$. Therefore, we can write the number of options for $A$ as $a_{k}(n)$, which is a polynomial of degree at most $2(k-2)^{2}$, which means there is some constant $\alpha_{k}$ such that we can bound $a_{k}(n)$ by $\alpha_{k}n^{2(k-2)^{2}}$

    For $B,C$, we also need to choose values, but our choice of indices is forced once we choose the indices for $A$. If there is a $b \in B$, $c \in C$ where $b$ is to the left of $c$, then either $b < c$ and $b$ has rev-rank at least $k$, or $c < b$ and $c$ has rank at least $k$, which we already established as not allowed. Thus, in choosing indices and values for $B,C$, there are $2^{n-s}$ options, as we only need to choose a subset of values to belong to $B$.

    Let $t$ be the size of $B$, therefore $n-s-t$ is the size of $C$. By earlier discussion, $B$ needs to avoid $132,(k-1)\dots 21$, and $C$ needs to avoid $231,12\dots(k-1)$. By Lemma 2.1, the number of ways to order the values of $B$ (resp $C$) is a $p_{k-1}(t)$, a polynomial in $t$ (resp $n-s-t$) of degree at most $2k-4$. So for sufficiently large $n$, we can bound each value by $\beta_{k} n^{2k-4}$.

    Therefore, the total number of ways to construct $\pi \in P_{n}$ can be written as:
    \begin{align*}
        |P_{n}| 
        \leq \sum_{s=1}^{(k-2)^{2}} a_{s}{n \choose s}^{2} 2^{n-s} (\beta_{k} n^{2k-4})^{2}
        \leq \big(\alpha_{k} \beta_{k}^{2}\big)2^{n} n^{2(k-2)^{2}}n^{4k-8}
        \leq \big(\alpha_{k} \beta_{k}^{2}\big)2^{n}n^{2k^{2}-4k}.
    \end{align*}
    Thus our set $P_{n}$ has exponential growth rate 2. Since this set is an upper bound on $\av(\sigma_{k},\sigma_{k}^{rev})$, we can thus conclude that $L(\sigma_{k},\sigma_{k}^{rev}) = 2$.
\end{proof}

Theorem 1.2 follows directly from Theorem 1.1, through a simple upper and lower bound of $L(\sigma,\sigma^{rev})$ and $L(\tau,\tau^{rev})$.

\end{section}

\begin{section}{Sub-Exponential Differences}

While our proof shows that the exponential growth rates of $(\sigma_{k},\sigma_{k}^{rev})$ are equivalent for all $k$, it should be clear that the actual values are not equivalent. The natural follow-up question would then be if the \textit{sub-exponential} growth rates are also the same. The bounds on sub-exponential growth used in the proof are $n^{0}$ for $\av(132,231)$, and $n^{2k^{2}-4k}$ for $|P_{n}|$, but these are far from optimal.

While we are unable to prove exact sub-exponential asymptotics, we are able to show their generating functions are rational, and they have different sub-exponential growth rates.

\begin{theorem}
    Let $k \geq 3$, $\sigma_{k} = 12\dots (k-2)k(k-1)$. Then the generating function for $\av(\sigma_{k},\sigma_{k}^{rev})$ is rational.
\end{theorem}

\begin{proof}
    This Theorem directly follows from Theorem 7 of \cite{vatter}. There, Vatter proved that a generating tree for $\av(\sigma,\sigma^{rev})$ with finitely many labels must have a rational generating function, and characterized the sets of patterns that fulfill that condition. He proved that a set of patterns must contain the child of an increasing pattern and the child of a decreasing pattern, defined (respectively) as patterns where the removal of the largest value would give the monotone increasing (decreasing) pattern. Our patterns $\big(\sigma_{k},\sigma_{k}^{rev}\big)$ fulfill these conditions, and thus our generating functions are rational.
\end{proof}

This theorem gives us some additional information about the asymptotics of our classes, and helps prove the following theorem:

\begin{theorem}
    Let $k \geq 3$. Let $\sigma = 12\dots (k-2)k(k-1), \tau = 12\dots (k-1)(k+1)k$. Then there exists $S,T \in \mathbb{R}$, $s < t \in \mathbb{N}$ such that:
    \begin{align*}
        \av(\sigma,\sigma^{rev}) \sim S\cdot2^{n}\cdot n^{s}, \\
        \av(\tau,\tau^{rev}) \sim T \cdot 2^{n} \cdot n^{t}.
    \end{align*}
    In particular, the sub-exponential growth rates of the sets are different positive integer powers of $n$.
\end{theorem}

\begin{proof}

First, $s,t \in \mathbb{N}$ follows directly from the generating functions for $\av(\sigma,\sigma^{rev})$ and $\av(\tau,\tau^{rev})$ being rational. To prove $s \neq t$, we will again use a bounding argument.

Assume $\pi \in \av(\sigma,\sigma^{rev})$. Observe that inserting the value of 1 anywhere in $\pi$, and increasing all entries of $\pi$ by 1, will result in a new permutation of $n+1$, and this resultant permutation will always avoid $\tau,\tau^{rev}$. If it did not, say it contained some $\tau$ pattern, then either that pattern uses the value of 1, or it doesn't. If it did, then the non-1 values of the $\tau$ pattern would form a $\sigma$ pattern in $\pi$. If it did not, then ignoring the smallest value of that $\tau$ pattern would also give a $\sigma$ pattern in $\pi$. A similar argument follows if the new permutation contains a $\tau^{rev}$.

There are $n+1$ options for where to insert a 1 in $\pi$, and all $\pi \in \av(\sigma,\sigma^{rev})$ and all choices of where to insert 1 give a different permutation. Therefore, $\text{Av}_{n+1}(\tau,\tau^{rev}) \geq (n+1)\av(\sigma,\sigma^{rev})$. Therefore $t \geq s+1$ and thus $t \neq s$.

\end{proof}

\end{section}

\begin{section}{Conclusion and Further Directions}

Let $I_{k} = 12\dots k, D_{k} = k\dots 21$. For two patterns $\sigma \in S_{k},\tau \in S_{m}$, we define the \textit{direct sum} $\oplus$ and \textit{skew sum} $\ominus$ as follows:
\begin{align*}
    (\sigma \oplus \tau)_{i} &=
        \begin{cases}
            \sigma_{i} & i \leq k, \\
            k+\tau_{k-i} & i > k,
        \end{cases}, \\
    (\sigma \ominus \tau)_{i} &= 
        \begin{cases}
            m+\sigma_{i} & i \leq k, \\
            \tau_{k-i} & i > k.
        \end{cases}
\end{align*}
In this language, we've shown that for $\sigma_{k} = I_{k-2} \oplus 21$, $L(\sigma,\sigma^{rev}) = 2$ for all $k \geq 3$.

In this paper we introduced a class of permutations whose exponential growth rate is independent of the patterns being avoided, similar to the monotone increasing and decreasing permutations of the Erd\H{o}s-Szekeres theorem. It is a natural question then, to ask when avoiding a pattern and its reverse gives such a result. In several cases we can prove an upper bound independent of pattern length, pointing towards an answer to this question:

\begin{enumerate}
    \item For $k \geq 4$, let $\pi_{k} = 21 \oplus I_{k-4} \oplus 21$. It can be proven using similar techniques to Theorem 1.1 that $4 \leq L(\pi_{k},\pi_{k}^{rev}) \leq 8$. Some numerical evidence suggests the true value is 8 for all $k$, but this is as yet unknown.
    \item For any pattern $\pi$, let $\pi'_{k} = I_{k} \oplus \pi$. It can be shown with naive techniques that $L(\pi'_{k},(\pi'_{k})^{rev}) \leq 4L(1\oplus \pi)$, providing an absolute bound on the growth rate independent of the size of $k$.
    \item As a generalization of both other points, for any pattern $\sigma,\tau$, let $\pi_{k} = \sigma \oplus I_{k} \oplus \tau$. Using methods similar to Theorem 4.2 in \cite{bona}, the value of $L(\pi_{k},\pi_{k}^{rev})$ can be bounded with an expression using only $L(1\oplus \tau),L(\sigma \oplus 1)$. However, the bounds obtained in this manner are often much larger than expected values, and serve primarily as an absolute bound for arbitrary $k$ rather than a plausible estimate.
\end{enumerate}

\end{section}

\section*{Declaration of AI Usage}

No AI was used at any point in this research, for either formulation or presentation of results.

\end{document}